\documentclass{amsproc}%
\usepackage{amsfonts}
\usepackage{amsmath}
\usepackage{amssymb}
\usepackage{graphicx}%
\providecommand{\U}[1]{\protect \rule{.1in}{.1in}}
\theoremstyle{plain}

\newtheorem{lemma}{Lemma}

\newtheorem{theorem}{Theorem}
\numberwithin{equation}{section}
\begin{document}
\title{Central Sequences in Subhomogeneous Unital C*-algebras}
\author{Don Hadwin}
\address{Mathematics Department, University of New Hampshire}
\email{don@unh.edu}
\author{Hemant Pendharkar}
\curraddr{Department of Mathematics and Statistics, University of South Florida}
\email{pendharkar@usf.edu}
\urladdr{math.usf.edu/faculty/hpendharkar}
\thanks{This paper is in final form and no version of it will be submitted for
publication elsewhere.}

\begin{abstract}
Suppose $\mathcal{A}$ is a unital subhomogeneous C*-algebra. We show that
every central sequence in $\mathcal{A}$ is hypercentral if and only if every
pointwise limit of a sequence of irreducible representations is multiplicity
free. We also show that every central sequence in $\mathcal{A}$ is trivial if
and only if every pointwise limit of irreducible representations is
irreducible. We also give a nice repesentation of the latter algebras.

\end{abstract}
\maketitle

The notion of central sequences (with respect to the $\left \Vert {}\right \Vert
_{2}$-norm defined by the trace) on $II_{1}$ factor von Neumann algebras has
played an important role in the study of these algebras. For example, D.
McDuff \cite{M} proved that if there are two central sequences in a type
$II_{1}$ factor von Neumann algebra $\mathcal{M}$ that do not asymptotically
commute, then $\mathcal{M}$ is isomorphic to $\mathcal{M}\otimes \mathcal{R}$,
where $\mathcal{R}$ is the hyperfinite $II_{1}$ factor. A $II_{1}$ factor has
property $\Gamma$ if and only if there is a central sequence that is bounded
away from the center.

The analogous notion for C*-algebras is more complicated. Suppose
$\mathcal{A}$ is a separable unital C*-algebra. A bounded sequence $\left \{
x_{n}\right \}  $ in $\mathcal{A}$ is called a \emph{central sequence} if, for
every $a\in \mathcal{A}$,%
\[
\lim_{n\rightarrow \infty}\left \Vert ax_{n}-x_{n}a\right \Vert =0\text{.}%
\]
A central sequence $\left \{  x_{n}\right \}  $ for $\mathcal{A}$ is called
\emph{hypercentral} if and only if, for every central sequence $\left \{
y_{n}\right \}  $,%
\[
\lim_{n\rightarrow \infty}\left \Vert x_{n}y_{n}-y_{n}x_{n}\right \Vert =0.
\]
Let $\mathcal{Z}\left(  \mathcal{A}\right)  $ denote the center of the
C*-algebra $\mathcal{A}$. A sequence $\left \{  x_{n}\right \}  $ is called a
\emph{trivial central sequence} if%
\[
\lim_{n\rightarrow \infty}\text{\textrm{dist}}\left(  x_{n},\mathcal{Z}\left(
\mathcal{A}\right)  \right)  =0.
\]
The study of central sequences and hypercentral sequences for a C*-algebra was
initiated by C. Akemann G. Pedersen \cite{AP} and J. Phillips \cite{P}, who
both related central sequences to the automorphism groups of the algebras.

The question of whether there are central sequences that are bounded away from
the center was solved by C. Akemann G. Pedersen \cite{AP}, who proved that
every central sequence in a separable C*-algebra $\mathcal{A}$ is trivial
exactly when $\mathcal{A}$ is a continuous trace C*-algebra. Thus most
C*-algebras have a large supply of nontrivial central sequences. The question
of when every two central sequences are asymptotically commuting, i.e., when
every central sequence is hypercentral has been partially solved. H. Ando and
E. Kirchberg \cite{AK} proved that if $\mathcal{A}$ is a separable C*-algebra
that is not type I (GCR), then there is central sequence that is not
hypercentral. However, if $\mathcal{A}$ is a unital type I separable
C*-algebra with an infinite-dimensional irreducible representation, then J.
Phillips \cite[Theorem 3.6]{P} proved that there is central sequence that is
not hypercentral. Thus if a unital separable C*-algebra has the property that
every two central sequences are asymptotically commuting, then every
irreducible representation must be finite-dimensional. In this paper we focus
on the class of \emph{subhomogeneous} C*-algebras, i.e., ones for with there
is a positive integer so that every irreducible representation is at most
$n$-dimensional. In this paper, for subhomogeneous separable unital
C*-algebras, we give characterizations, in terms of finite-dimensional
irreducible representations, of the property that every central sequence is
trivial (Theorem \ref{ct}), and of the property that every central sequence is
hypercentral (Theorem \ref{hc}). Thus, for subhomogeneous C*-algebras our
results give a nice comparison of these two properties. Most of the results in
this paper appeared in the second author's PhD dissertation \cite{HP}.

\textbf{NOTE:} We have been informed by Tatiana Shulman, that she and Dominic
Enders have completely characterized all C*-algebras whose central sequences
are all hypercentral. 

The following is contained in \cite{AP} and \cite{P}.

\begin{lemma}
If $\mathcal{A}$ is a separable C*-algebra, then
\end{lemma}

\begin{enumerate}
\item If $\pi:\mathcal{A}\mathbf{\rightarrow}\mathcal{B}$ is a surjective
unital $\ast$-homomorphism and $\left \{  b_{n}\right \}  $ is a central
sequence in $\mathcal{B}$, then there is a central sequence $\left \{
a_{n}\right \}  $ in $\mathcal{A}$ such that $\pi \left(  a_{n}\right)  =b_{n}$
for every $n\in \mathbb{N}$.

\item Every central sequence of $\mathcal{A}$ is trivial if and only if
$\mathcal{A}$ is a continuous trace C*-algebra.

\item If $\pi:\mathcal{A}\mathbf{\rightarrow}\mathcal{B}$ is a surjective
unital $\ast$-homomorphism and every central sequence in $\mathcal{A}$ is
hypercentral, then every central sequence of $\mathcal{B}$ is hypercentral.
\end{enumerate}

Suppose $\mathcal{A}$ is a unital C*-algebra and $n$ is a positive integer.
Let \textrm{Rep}$_{n}\left(  \mathcal{A}\right)  $ be the set of unital $\ast
$-homomorphisms $\pi:\mathcal{A}\rightarrow \mathbb{M}_{n}\left(
\mathbb{C}\right)  $, with the topology of pointwise convergence. We define
\textrm{Irr}$_{n}\left(  \mathcal{A}\right)  $ to be the set of irreducible
representations in \textrm{Rep}$_{n}\left(  \mathcal{A}\right)  $. Every
representation $\pi \in$ \textrm{Rep}$_{n}\left(  \mathcal{A}\right)  $ can be
written as a direct sum%
\[
\pi=\sigma_{1}\oplus \cdots \oplus \sigma_{k}%
\]
with each $\sigma_{j}$ irreducible. We say that $\pi \in$ \textrm{Rep}%
$_{n}\left(  \mathcal{A}\right)  $ is \emph{multiplicity free} if, whenever
$1\leq i<j\leq k$, $\sigma_{i}$ and $\sigma_{j}$ are not unitarily equivalent.
Equivalent, $\pi$ is multiplicity free if and only if $\pi \left(
\mathcal{A}\right)  ^{\prime}$ is abelian.

If $\mathcal{A}$ is separable and $\left \{  x_{1},x_{2},\ldots \right \}  $ is
dense in the unit ball of $\mathcal{A}$, then%
\[
d\left(  \pi,\rho \right)  =\sum_{k=1}^{\infty}\frac{1}{2^{k}}\left \Vert
\pi \left(  x_{k}\right)  -\rho \left(  x_{k}\right)  \right \Vert
\]
is a metric on \textrm{Rep}$_{n}\left(  \mathcal{A}\right)  $ that makes a
compact metric space.

A C*-algebra $\mathcal{A}$ is \emph{subhomogeneous} if there is a smallest
positive integer $N$ such that every irreducible representation of
$\mathcal{A}$ is unitarily equivalent to a representation in
\[
\cup_{1\leq k\leq N}\text{\textrm{Irr}}_{k}\left(  \mathcal{A}\right)  \text{.
}%
\]
We call $N$ the \emph{degree} of $\mathcal{A}$. It was shown by L. Robert
\cite{LR} that a unital continuous trace C*-algebra is subhomogeneous.

\begin{lemma}
\label{useful}Suppose $\mathcal{A}$ is a unital separable C*-algebra,
$s\in \mathbb{N}$, $\left \{  \pi_{n}\right \}  $ is a sequence in \textrm{Irr}%
$_{s}\left(  \mathcal{A}\right)  $ such that, $\pi_{m}$ and $\pi_{n}$ are not
unitarily equivalent whenever $1\leq m<n<\infty$. Suppose $\pi_{0}\in$
\textrm{Rep}$_{s}\left(  \mathcal{A}\right)  ,$ and $\pi_{n}\rightarrow \pi
_{0}$ pointwise on $\mathcal{A}$, and, for every $n\in \mathbb{N}$, $\pi_{0}$
is not unitarily equivalent to $\pi_{n}$. Let $\pi=\pi_{0}\oplus \pi_{1}%
\oplus \pi_{2}\oplus \cdots$. Then

\begin{enumerate}
\item $T=T_{0}\oplus T_{1}\oplus T_{2}\oplus \cdots \in \pi \left(  \mathcal{A}%
\right)  $ if and only if $T_{0}\in \pi_{0}\left(  \mathcal{A}\right)  $ and
\[
\lim_{n\rightarrow \infty}\left \Vert T_{n}-T_{0}\right \Vert =0\text{.}%
\]

\item $A\in \pi_{0}\left(  \mathcal{A}\right)  ^{\prime}$ in $\mathbb{M}%
_{s}\left(  \mathcal{A}\right)  $ if and only if there is a central sequence
$\left \{  a_{n}\right \}  $ in $\mathcal{A}$ and $\pi_{n}\left(  a_{n}\right)
\rightarrow A\in \mathbb{M}_{s}\left(  \mathbb{C}\right)  $.
\end{enumerate}
\end{lemma}

\begin{proof}
$\left(  1\right)  $. Since $\pi \left(  a\right)  =\pi_{0}\left(  a\right)
\oplus \pi_{1}\left(  a\right)  \oplus \pi_{2}\left(  a\right)  \oplus \cdots$
and $\pi_{n}\left(  a\right)  \rightarrow \pi_{0}\left(  a\right)  ,$ we see
that the condition is necessary. Since, for $0\leq m<n<\infty$, $\pi_{n}$ and
$\pi_{n}$ are finite-dimensional representations with no unitarily equivalent
irreducible direct summands, there are elements $a_{mn}\in \mathcal{A}$ with
$0\leq a_{mn}\leq1$ such that $\pi_{m}\left(  a_{mn}\right)  =0$ and $\pi
_{n}\left(  a_{mn}\right)  =1.$ Thus $\lim_{n\rightarrow \infty}\left \Vert
\pi_{n}\left(  a_{01}\right)  \right \Vert =\left \Vert \pi_{0}\left(
a_{01}\right)  \right \Vert =0$. Hence there is an $n_{0}\in \mathbb{N}$ such
that $0\leq \pi_{n}\left(  a_{01}\right)  \leq1/2$ whenever $n\geq n_{0}$. If
$f:\left[  0,1\right]  \rightarrow \left[  0,1\right]  $ is continuous with
$f\left(  t\right)  =0$ when $t\in \left[  0,1/2\right]  $ and $f\left(
1\right)  =1,$ then%
\[
\pi_{1}\left(  f\left(  a_{01}\right)  \right)  =1\text{ and }\pi_{n}\left(
f\left(  a_{01}\right)  \right)  =0\text{ when }n\geq n_{0}\text{.}%
\]
Thus $p_{1}=f\left(  a_{01}\right)
{\displaystyle \prod_{k=2}^{n_{0}-1}}
a_{k1}$ has the property that $\pi_{1}\left(  p_{1}\right)  =1$ and $\pi
_{n}\left(  p_{1}\right)  =0$ when $n\neq1.$ Similarly, for $n=2,3,\ldots$
there is a $p_{n}\in \mathcal{A}$ such that $\pi_{n}\left(  p_{n}\right)  =1$
and $\pi_{k}\left(  p_{n}\right)  =0$ when $k\neq n$. Since $\pi
_{n}:\mathcal{A}\rightarrow \mathbb{M}_{n}\left(  \mathbb{C}\right)  $ is
surjective for each $n\in \mathbb{N}$, we see that the set $\mathcal{S}$ of all
$0\oplus T_{1}\oplus T_{2}\oplus \cdots$ with only finitely many nonzero
summands is contained in $\pi \left(  \mathcal{A}\right)  $. But the closure
$\mathcal{S}^{-}$ of $\mathcal{S}$ is the set of all $0\oplus T_{1}\oplus
T_{2}\oplus \cdots$ with $\lim_{n\rightarrow \infty}\left \Vert T_{n}\right \Vert
=0$ is contained in $\pi \left(  \mathcal{A}\right)  $. If $T_{0}=\pi
_{0}\left(  a\right)  $ for some $a\in \mathcal{A}$ and $\lim_{n\rightarrow
\infty}T_{n}=T_{0}$, then
\[
T_{0}\oplus T_{1}\oplus \cdots-\pi \left(  a\right)  \in \mathcal{S}^{-}%
\subset \pi \left(  \mathcal{A}\right)  ,
\]
so
\[
T_{0}\oplus T_{1}\oplus \cdots \in \pi \left(  \mathcal{A}\right)  \text{.}%
\]

$\left(  2\right)  $. Suppose $\left \{  a_{n}\right \}  $ is a central sequence
and $\pi_{n}\left(  a_{n}\right)  \rightarrow A\in \mathbb{M}_{s}\left(
\mathbb{C}\right)  $. Suppose $T\in \pi_{0}\left(  \mathcal{A}\right)  $. Then,
by $\left(  1\right)  $, $\hat{T}=T\oplus T\oplus \cdots \in \pi \left(
\mathcal{A}\right)  $. Thus%
\[
\left \Vert AT-TA\right \Vert =\lim_{n\rightarrow \infty}\left \Vert \pi
_{n}\left(  a_{n}\right)  T-T\pi_{n}\left(  a_{n}\right)  \right \Vert
\]%
\[
\leq \lim_{n}\left \Vert \pi \left(  a_{n}\right)  \hat{T}-\hat{T}\pi \left(
a_{n}\right)  \right \Vert =0.
\]
Hence $A\in \pi_{0}\left(  \mathcal{A}\right)  ^{\prime}$. Now suppose $A\in
\pi_{0}\left(  \mathcal{A}\right)  ^{\prime}$, and, for each $n\in \mathbb{N}$,
define $A_{n}=A_{n}\left(  0\right)  \oplus A_{n}\left(  1\right)  \oplus
A_{n}\left(  2\right)  \oplus \cdots$, where $A_{n}\left(  n\right)  =A$ and
$A_{n}\left(  k\right)  =0$ when $k\neq n$. We know from $\left(  1\right)  $
that $A_{n}\in \pi \left(  T\right)  $ and, for every $T\in \pi_{0}\left(
\mathcal{A}\right)  $ that%
\[
A_{n}\hat{T}-\hat{T}A_{n}=0.
\]
Thus $\left \{  A_{n}\right \}  $ is a central sequence in $\pi \left(
\mathcal{A}\right)  $ and there is a central sequence $\left \{  a_{n}\right \}
$ in $\mathcal{A}$ such that, for each $n\in \mathbb{N}$, $\pi \left(
a_{n}\right)  =A_{n}$ and $\pi_{n}\left(  a_{n}\right)  =A$.
\end{proof}

\begin{theorem}
\label{hc}Suppose $\mathcal{A}$ is a separable unital C*-algebra. Then

\begin{enumerate}
\item If every central sequence of $\mathcal{A}$ is hypercentral, then, for
every $s\in \mathbb{N}$, every representation in \textrm{Irr}$_{s}\left(
\mathcal{A}\right)  ^{-}$ is multiplicity free.

\item If $\mathcal{A}$ is subhomogeneous, and, for every $s\in \mathbb{N}$,
every representation in \textrm{Irr}$_{s}\left(  \mathcal{A}\right)  ^{-}$ is
multiplicity free, then every central sequence in $\mathcal{A}$ is hypercentral.
\end{enumerate}
\end{theorem}

\begin{proof}
$\left(  1\right)  .$ Suppose every central sequence of $\mathcal{A}$ is
hypercentral. Assume, via contradiction, that there is an $s$ in $\mathbb{N}$
and a sequence $\left \{  \pi_{k}\right \}  $ in \textrm{Irr}$_{s}\left(
\mathcal{A}\right)  $ such that, for some $\pi \in$ \textrm{Rep}$_{s}\left(
\mathcal{A}\right)  $, $\pi_{k}\left(  a\right)  \rightarrow \pi \left(
a\right)  $ for every $a\in A$, and $\pi$ is not multiplicity free. Thus
$\pi \left(  \mathcal{A}\right)  ^{\prime}$ contains projections $P$ and $Q$
such that $PQ-QP\neq0.$ Let
\[
\pi_{\infty}=\pi \oplus \pi_{1}\oplus \pi_{2}\oplus \cdots \text{ .}%
\]
If infinitely many of the $\pi_{n}$'s are unitarily equivalent, we can replace
the sequence with a new sequence such that, for every $n\in \mathbb{N}$, there
is a unitary matrix $U_{n}$ such that $\pi_{n}\left(  a\right)  =U_{n}\pi
_{1}\left(  a\right)  U_{n}^{\ast}$ for every $a\in \mathcal{A}$. In this case,
if we find a subsequence $\left \{  U_{n_{k}}\right \}  $ that converges to a
unitary $U$, we have $\pi_{0}$ is unitarily equivalent to $\pi_{1}$, which
contradicts the assumption that $\pi_{0}$ is multiplicity free. Thus we can
assume that, when $0\leq m<n<\infty$, $\pi_{m}$ and $\pi_{n}$ are not
unitarily equivalent. Suppose $A,B\in \pi_{0}\left(  \mathcal{A}\right)
^{\prime}$. By Lemma \ref{useful}, there are central sequences $\left \{
a_{n}\right \}  ,\left \{  b_{n}\right \}  $ in $\mathcal{A}$ such that%
\[
\lim_{n\rightarrow \infty}\pi_{n}\left(  a_{n}\right)  =A\text{ and }%
\lim_{n\rightarrow \infty}\pi_{n}\left(  b_{n}\right)  =B.
\]
Since $\left \{  a_{n}\right \}  $ is hypercentral,
\[
\left \Vert AB-BA\right \Vert =\lim_{n}\left \Vert \pi_{n}\left(  a_{n}\right)
\pi_{n}\left(  b_{n}\right)  -\pi_{n}\left(  b_{n}\right)  \pi_{n}\left(
a_{n}\right)  \right \Vert
\]%
\[
\leq \lim_{n}\left \Vert \pi \left(  a_{n}\right)  \pi \left(  b_{n}\right)
-\pi \left(  b_{n}\right)  \pi \left(  a_{n}\right)  \right \Vert =0.
\]
Thus $\pi_{0}\left(  \mathcal{A}\right)  ^{\prime}$ is abelian, which means
$\pi_{0}$ is multiplicity free.

$\left(  2\right)  .$ Suppose $\mathcal{A}$ is subhomogeneous. Then there is
an $N$ such that every irreducible representation of $\mathcal{A}$ is
unitarily equivalent to a representation in $\cup_{s=1}^{N}$\textrm{Irr}%
$_{s}\left(  \mathcal{A}\right)  $. Now suppose that there is a central
sequence $\left \{  a_{n}\right \}  $ in $\mathcal{A}$ that is not hypercentral.
Then there is a subsequence $\left \{  a_{n_{k}}\right \}  $ and a central
sequence $\left \{  b_{n}\right \}  $ in $\mathcal{A}$ and an $\varepsilon>0$
such that, for every $k\in \mathbb{N}$,%
\[
\left \Vert a_{n_{k}}b_{n_{k}}-b_{n_{k}}a_{n_{k}}\right \Vert \geq \varepsilon.
\]
By relabeling the subsequences, we can assume, for every $n\in \mathbb{N}$,
that
\[
\left \Vert a_{n}b_{n}-b_{n}a_{n}\right \Vert \geq \varepsilon \text{ .}%
\]
For each $n\in \mathbb{N}$ there is a representation $\pi_{n}\in \cup_{s=1}^{N}%
$\textrm{Irr}$_{s}\left(  \mathcal{A}\right)  $ such that, for every
$n\in \mathbb{N}$,%
\[
\left \Vert \pi_{n}\left(  a_{n}b_{n}-b_{n}a_{n}\right)  \right \Vert
=\left \Vert a_{n}b_{n}-b_{n}a_{n}\right \Vert \text{.}%
\]
By again restricting to a subsequence we can assume that there is an $s,$
$1\leq s\leq N$, such that $\pi_{n}\in$ \textrm{Irr}$_{s}\left(
\mathcal{A}\right)  $. Since \textrm{Rep}$_{s}\left(  \mathcal{A}\right)  $ is
compact, we can also assume that there is a $\pi_{0}\in$ \textrm{Rep}%
$_{s}\left(  \mathcal{A}\right)  $ such that $\pi_{n}\left(  a\right)
\rightarrow \pi_{0}\left(  a\right)  $ for every $a\in \mathcal{A}$. If $\pi
_{0}$ is unitarily equivalent to infinitely many of $\pi_{1},\pi_{2},\ldots,$
we have for infinitely many $k\in \mathbb{N}$ that $\left \Vert \pi_{0}\left(
a_{n_{k}}b_{n_{k}}-b_{n_{k}}a_{n_{k}}\right)  \right \Vert \geq \varepsilon$ and
$\left \{  \pi_{0}\left(  a_{n_{k}}\right)  \right \}  $ and $\left \{  \pi
_{0}\left(  b_{n_{k}}\right)  \right \}  $ are central sequence for $\pi
_{0}\left(  \mathcal{A}\right)  =\mathcal{M}_{s}\left(  \mathbb{C}\right)  $,
which is impossible. Hence we can assume that, for every $n\in \mathbb{N}$,
$\pi_{0}$ is not unitarily equivalent to $\pi_{n}$. Finally, we can assume
that
\[
\pi_{n}\left(  a_{n}\right)  \rightarrow A\text{ and }\pi_{n}\left(
b_{n}\right)  \rightarrow B
\]
in $\mathbb{M}_{n}\left(  \mathbb{C}\right)  $. By Lemma \ref{useful},
$A,B\in \pi_{0}\left(  \mathcal{A}\right)  ^{\prime}$. Hence $\pi_{0}\left(
\mathcal{A}\right)  ^{\prime}$ is not abelian, and
\[
\left \Vert AB-BA\right \Vert =\lim_{n}\left \Vert \pi_{n}\left(  a_{n}\right)
\pi_{n}\left(  b_{n}\right)  -\pi_{n}\left(  b_{n}\right)  \pi_{n}\left(
a_{n}\right)  \right \Vert \geq \varepsilon \text{.}%
\]
Thus $\pi_{0}\left(  \mathcal{A}\right)  ^{\prime}$ is not abelian, so
$\pi_{0}$ is not multiplicity free.
\end{proof}

\bigskip

To compare the property that every central sequence in a subhomogeneous
C*-algebra is hypercentral with the property that every central sequence is
trivial we will show that the latter is equivalent to \textrm{Irr}$_{s}\left(
\mathcal{A}\right)  $ is closed for every $s\in \mathbb{N}$.

Suppose $s$ is a positive integer. We define an $s$\emph{-equivalence system}
to be a tuple $\left(  X,\alpha,\beta \right)  $ where

\begin{enumerate}
\item $X$ is a compact metric space,

\item $\alpha \subset X\times X$ is a closed equivalence relation,

\item $\beta:\alpha \rightarrow \mathcal{U}_{s}$ (the set of unitary operators
in $\mathbb{M}_{s}\left(  \mathbb{C}\right)  $) such that, whenever $\left(
x,y\right)  ,\left(  y,z\right)  \in X$%
\[
\beta \left(  x,y\right)  ^{\ast}\beta \left(  x,z\right)  =\beta \left(
y,z\right)  \text{.}%
\]

\end{enumerate}

If $\left(  X,\alpha,\beta \right)  $ is an $s$-equivalence system, we define
$C\left(  X,\alpha,\beta,\mathbb{M}_{s}\left(  \mathbb{C}\right)  \right)  $
to be the set of all functions $f\in C\left(  X,\mathbb{M}_{s}\left(
\mathbb{C}\right)  \right)  $ such that, whenever $x,y\in \alpha$,%
\[
\beta \left(  x,y\right)  f\left(  x\right)  \beta \left(  x,y\right)  ^{\ast
}=f\left(  y\right)  .
\]
It is clear that $C\left(  X,\alpha,\beta,\mathbb{M}_{s}\left(  \mathbb{C}%
\right)  \right)  $ is a unital C*-algebra that contains every function of the
form%
\[
f\left(  x\right)  =h\left(  x\right)  I_{s}%
\]
with $h\in C\left(  X\right)  $ such that $h\left(  x\right)  =h\left(
y\right)  $ whenever $\left(  x,y\right)  \in \alpha$. We say that the system
$\left(  X,\alpha,\beta \right)  $ is \emph{regular} if, for every $x\in X$,%
\[
\left \{  f\left(  x\right)  :f\in C\left(  X,\alpha,\beta,\mathbb{M}%
_{s}\left(  \mathbb{C}\right)  \right)  \right \}  =\mathbb{M}_{s}\left(
\mathbb{C}\right)  \text{.}%
\]

\begin{lemma}
\label{triv}If $\left(  X,\alpha,\beta \right)  $ is a regular $s$-equivalence
system, then every central sequence in $C\left(  X,\alpha,\beta,\mathbb{M}%
_{s}\left(  \mathbb{C}\right)  \right)  $ is trivial.
\end{lemma}

\begin{proof}
Suppose $\left \{  a_{n}\right \}  $ is a central sequence in $C\left(
X,\alpha,\beta,\mathbb{M}_{s}\left(  \mathbb{C}\right)  \right)  $.

\textbf{Claim:}
\[
\lim_{n\rightarrow \infty}\sup_{x\in X}\text{\textrm{dist}}\left(  a_{n}\left(
x\right)  ,\mathbb{C}I_{s}\right)  =0.
\]

Assume, via contradiction, that the claim is false. By considering a
subsequence, we can assume that there is an $\varepsilon>0$, and a sequence
$\left \{  x_{n}\right \}  $ in $X$ such that,%

\[
\lim_{n\rightarrow \infty}x_{n}=x_{0}\in X\text{ and }a\left(  x_{n}\right)
\rightarrow A\in \mathbb{M}_{s}\left(  \mathbb{C}\right)  ,
\]
and for every $n\in \mathbb{N}$,%
\[
\text{\textrm{dist}}\left(  a_{n}\left(  x_{n}\right)  ,\mathbb{C}%
I_{s}\right)  \geq \varepsilon \text{.}%
\]
It follows that \textrm{dist}$\left(  A,\mathbb{C}I_{s}\right)  \geq
\varepsilon.$ Hence there is a $B\in \mathbb{M}_{s}\left(  \mathbb{C}\right)  $
such that $\left \Vert AB-BA\right \Vert =\delta>0$. We can choose $f\in
C\left(  X,\alpha,\beta,\mathbb{M}_{s}\left(  \mathbb{C}\right)  \right)  $ so
that $f\left(  x_{0}\right)  =B$. Thus
\[
0=\lim_{n\rightarrow \infty}\left \Vert fa_{n}-a_{n}f\right \Vert \geq
\lim_{n\rightarrow \infty}\left \Vert f\left(  x_{n}\right)  a_{n}\left(
x_{n}\right)  -a_{n}\left(  x_{n}\right)  f\left(  x_{n}\right)  \right \Vert
\]%
\[
=\left \Vert f\left(  x_{0}\right)  A-Af\left(  x_{0}\right)  \right \Vert
\geq \delta \text{.}%
\]
This contradiction proves the claim.

If $T\in \mathbb{M}_{s}\left(  \mathbb{C}\right)  $ and $\lambda \in \mathbb{C}$
and $\left \Vert T-\lambda I_{s}\right \Vert =$ \textrm{dist}$\left(
T,\mathbb{C}I_{s}\right)  ,$ then%
\[
\left \vert \tau_{s}\left(  T\right)  -\lambda \right \vert =\left \Vert \tau
_{s}\left(  T-\lambda I_{s}\right)  \right \Vert \leq \mathrm{dist}\left(
T,\mathbb{C}I_{s}\right)  \text{.}%
\]
Thus $\left \Vert T-\tau_{s}\left(  T\right)  \right \Vert \leq2$\textrm{dist}%
$\left(  T,\mathbb{C}I_{s}\right)  .$ For each $n\in \mathbb{N}$, define
$z_{n}=\tau_{s}\circ a_{n}$. Clearly, $z_{n}\in \mathcal{Z}\left(  C\left(
X,\alpha,\beta,\mathbb{M}_{s}\left(  \mathbb{C}\right)  \right)  \right)  ,$
and%
\[
\lim_{n\rightarrow \infty}\left \Vert a_{n}-z_{n}\right \Vert =\lim
_{n\rightarrow \infty}\sup_{x\in X}\left \Vert a_{n}\left(  x\right)  -\tau
_{s}\left(  a_{n}\left(  x\right)  \right)  I_{s}\right \Vert
\]%
\[
\leq2\lim_{n\rightarrow \infty}\sup_{x\in X}\text{\textrm{dist}}\left(
a_{n}\left(  x\right)  ,\mathbb{C}I_{s}\right)  =0\text{.}%
\]
Hence $\left \{  a_{n}\right \}  $ is trivial.
\end{proof}

\bigskip

\begin{lemma}
\label{iso}Suppose $\mathcal{A}$ is a separable unital $C^{\ast}$-algebra and
$s\in \mathbb{N}$, \textrm{Irr}$_{s}\left(  \mathcal{A}\right)  $ is closed,
and every irreducible representation of $\mathcal{A}$ is unitarily equivalent
to a representation in \textrm{Irr}$_{s}\left(  \mathcal{A}\right)  $. Then
there is a regular $s$-equivalence system $\left(  X,\alpha,\beta \right)  $
such that $\mathcal{A}$ is isomorphic to $C\left(  X,\alpha,\beta
,\mathbb{M}_{s}\left(  \mathbb{C}\right)  \right)  $.
\end{lemma}

\begin{proof}
Let $X=$ \textrm{Irr}$_{s}\left(  \mathcal{A}\right)  $ and let $\alpha$
denote unitary equivalence. For each $\pi_{0}\in X$ and each $\pi$ such that
$\pi$ $\alpha$ $\pi_{0}$ we can choose a unitary $\beta \left(  \pi_{0}%
,\pi \right)  =U_{\left(  \pi_{0},\pi \right)  }$ such that
\[
U_{\left(  \pi_{0},\pi \right)  }\pi_{0}\left(  \cdot \right)  U_{\left(
\pi_{0},\pi \right)  }^{\ast}=\pi \left(  \cdot \right)  .
\]
If $\rho \in X$ and $\rho$ $\alpha$ $\pi$, we define $U_{\left(  \pi
,\rho \right)  }=U_{\left(  \pi_{0},\rho \right)  }U_{\left(  \pi_{0}%
,\pi \right)  }^{\ast}$. We define
\[
\Gamma:\mathcal{A}\rightarrow C\left(  X,\alpha,\beta,\mathbb{M}_{s}\left(
\mathbb{C}\right)  \right)
\]
by%
\[
\Gamma \left(  a\right)  \left(  \pi \right)  =\pi \left(  a\right)  .
\]
It follows that $\beta$ is regular. A pure state on $C\left(  X,\alpha
,\beta,\mathbb{M}_{s}\left(  \mathbb{C}\right)  \right)  $ can be extended to
a pure state $\varphi$ on $C\left(  X,\mathbb{M}_{s}\left(  \mathbb{C}\right)
\right)  $, which has the form%
\[
\varphi \left(  f\right)  =\left \langle f\left(  x_{0}\right)
e,e\right \rangle
\]
for some $x_{0}\in X$, which restricted to $C\left(  X,\alpha,\beta
,\mathbb{M}_{s}\left(  \mathbb{C}\right)  \right)  $ is a pure state. It is
clear that the set of pure states on $C\left(  X,\alpha,\beta,\mathbb{M}%
_{s}\left(  \mathbb{C}\right)  \right)  $ is closed and that $\Gamma \left(
\mathcal{A}\right)  $ separates the pure states on $\mathcal{A}$. It follows
from Glimm's Stone-Weierstrass theorem \cite{G} that $\Gamma \left(
\mathcal{A}\right)  =C\left(  X,\alpha,\beta,\mathbb{M}_{s}\left(
\mathbb{C}\right)  \right)  $.
\end{proof}

\bigskip

\begin{theorem}
\label{ct} Suppose $\mathcal{A}$ is a separable unital subhomogeneous
C*-algebra. The following are equivalent:

\begin{enumerate}
\item $\mathcal{A}$ is a continuous trace C*-algebra.

\item Every central sequence in $\mathcal{A}$ is trivial.

\item For every $s\in \mathbb{N}$, \textrm{Irr}$_{s}\left(  \mathcal{A}\right)
$ is closed in \textrm{Rep}$_{s}\left(  \mathcal{A}\right)  $.

\item $\mathcal{A}$ is isomorphic to a finite direct sum of C*-algebras of the
form \newline$C\left(  X,\alpha,\beta,\mathbb{M}_{s}\left(  \mathbb{C}\right)
\right)  $ with $\left(  X,\alpha,\beta \right)  $ a regular $s$-equivalence system.
\end{enumerate}
\end{theorem}

\begin{proof}
$\left(  1\right)  \Leftrightarrow \left(  2\right)  $. This is proved in
\cite{AP}.

$\left(  2\right)  \Rightarrow \left(  3\right)  $. Suppose \textrm{Irr}%
$_{s}\left(  \mathcal{A}\right)  $ is not closed. Then there is a $\pi_{0}\in$
\textrm{Rep}$_{s}\left(  \mathcal{A}\right)  $ that is not irreducible and a
sequence $\left \{  \pi_{n}\right \}  $ in \textrm{Irr}$_{s}\left(
\mathcal{A}\right)  $ such that $\pi_{n}\left(  a\right)  \rightarrow \pi
_{0}\left(  a\right)  $. Arguing as in the proof of Theorem \ref{hc}$\left(
1\right)  $, we can assume that, for $1\leq m<n<\infty$, $\pi_{n}$ and
$\pi_{n}$ are not unitarily equivalent. Let $\pi=\pi_{0}\oplus \pi_{1}%
\oplus \cdots$. Since $\pi_{0}$ is not irreducible, there is a projection $P$
with $0\neq P\neq1$ such that $P\in \pi_{0}\left(  \mathcal{A}\right)
^{\prime}$. Thus, by Lemma \ref{useful}, there is a central sequence $\left \{
a_{n}\right \}  $ in $\mathcal{A}$ such that
\[
\pi_{n}\left(  a_{n}\right)  \rightarrow P.
\]
Thus
\[
\frac{1}{2}=\text{\textrm{dist}}\left(  P,\mathbb{C}I_{s}\right)
=\lim_{n\rightarrow \infty}\text{\textrm{dist}}\left(  \pi_{n}\left(
a_{n}\right)  ,\mathbb{C}I_{r}\right)
\]%
\[
\leq \lim_{n\rightarrow \infty}\text{\textrm{dist}}\left(  a_{n},\mathcal{Z}%
\left(  \mathcal{A}\right)  \right)  =0.
\]
This contradiction proves $\left(  2\right)  \Rightarrow \left(  3\right)  $.

$\left(  3\right)  \Rightarrow \left(  4\right)  $. Suppose \textrm{Irr}%
$_{s}\left(  \mathcal{A}\right)  $ is closed in \textrm{Rep}$_{s}\left(
\mathcal{A}\right)  $ for every $s\in \mathbb{N}$. Since $\mathcal{A}$ is
subhomogeneous, there is a minimal positive integer $N$ such that every
irreducible representation of $\mathcal{A}$ is unitarily equivalent to a
representation in $\cup_{s=1}^{N}$\textrm{Irr}$_{s}\left(  \mathcal{A}\right)
$. Suppose $1\leq m<n\leq N$. For each $\pi \in$\textrm{Irr}$_{m}\left(
\mathcal{A}\right)  $ and each $\rho \in$\textrm{Irr}$_{n}\left(
\mathcal{A}\right)  $ there is an $a\in \mathcal{A}$ with $0\leq a\leq1$ such
that $\pi \left(  a\right)  =0$ and $\rho \left(  a\right)  =1$. Thus there is
an open subset $U_{\pi,\rho}$ containing $\pi$ such that, for every $\sigma \in
U$, $0\leq \sigma \left(  a\right)  \leq1/2$. Choose a continuous $h:\left[
0,1\right]  \rightarrow \left[  0,1\right]  $ so that $h\left(  1\right)  $ and
$h|_{\left[  0,1/2\right]  }=0.$ Then $b_{U,\pi,\rho}=h\left(  a\right)  $
satisfies $\sigma \left(  b_{U,\pi,\rho}\right)  =0$ for every $\sigma \in
U_{\pi,\rho}$ and $\rho \left(  b_{U,\rho}\right)  =1$. Since $\left \{
U_{\pi,\rho}:\pi \in \text{\textrm{Irr}}_{m}\left(  \mathcal{A}\right)
\right \}  $ there is a finite subcover $\left \{  U_{\pi_{k},\rho}:1\leq k\leq
t\right \}  $ so if $b_{m,\rho}=\left(
{\displaystyle \prod_{1\leq i\leq t}}
b_{U_{\pi_{i},\rho}}\right)  ^{\ast}\left(
{\displaystyle \prod_{1\leq i\leq t}}
b_{U_{\pi_{i},\rho}}\right)  $, then $0\leq b_{m,\rho}\leq1$ and $\pi \left(
b_{m,\rho}\right)  =0$ for every $\pi \in$ \textrm{Irr}$_{m}\left(
\mathcal{A}\right)  $ and $\rho \left(  b_{m,\rho}\right)  =1.$ We now let
$c_{m,r}=1-b_{m,\rho}$. Then $\pi \left(  c_{m,r}\right)  =1$ for every $\pi
\in$ \textrm{Irr}$_{m}\left(  \mathcal{A}\right)  $ and $0\leq \sigma \left(
c_{m,\rho}\right)  \leq1/2.$ Proceeding as before, there is a $d_{m,n}%
\in \mathcal{A}$ such that $0\leq d_{m,n}\leq1$ and $\pi \left(  d_{m,n}\right)
=1$ for $\pi \in$ \textrm{Irr}$_{m}\left(  \mathcal{A}\right)  $ and
$\sigma \left(  d_{m,n}\right)  =1$ for $\sigma \in$\textrm{Irr}$_{n}\left(
\mathcal{A}\right)  .$ It follows that, for $1\leq m\leq N$ there is a
$p_{m}\in \mathcal{A}$ with $0\leq p_{m}\leq1$ such that $\pi \left(
p_{m}\right)  =1$ when $\pi \in$\textrm{Irr}$_{m}\left(  \mathcal{A}\right)  $
and $\pi \left(  p_{n}\right)  =0$ when $\pi \in$ \textrm{Irr}$_{k}\left(
\mathcal{A}\right)  $ with $k\neq m.$ Thus $p_{n}$ is a central projection and
$p_{1}+\cdots+p_{N}=1$. Thus $\mathcal{A}$ is a direct sum
\[
\mathcal{A}=\mathcal{A}_{1}\oplus \cdots \oplus \mathcal{A}_{N}%
\]
where, for $1\leq m\leq N$, every irreducible representation of $\mathcal{A}%
_{m}$ is unitarily equivalent to a representation in \textrm{Irr}$_{m}\left(
\mathcal{A}_{m}\right)  $ and \textrm{Irr}$_{m}\left(  \mathcal{A}_{m}\right)
$ is closed. Thus, by Lemma \ref{iso}, $\mathcal{A}_{m}$ is isomorphic to
$C\left(  X_{m},\alpha_{m},\beta_{m},\mathbb{M}_{m}\left(  \mathbb{C}\right)
\right)  $ for some regular $m$-equivalence system $\left(  X_{m},\alpha
_{m},\beta_{m}\right)  $.

$\left(  4\right)  \Rightarrow \left(  1\right)  $. This follows from Lemma
\ref{triv}.
\end{proof}

\end{document}